\documentclass[11pt,twoside]{amsart}       

\usepackage{graphicx}
\usepackage{amssymb,amsmath,amscd,amsfonts}
\usepackage{latexsym,color,epsfig,pgf,cite}
\usepackage{verbatim}

\newcommand{\enne}{\mathbb{N}}

\newcommand{\dint}{\displaystyle{\int}}

\newtheorem{theorem}{Theorem}[section]

\newtheorem{proposition}[theorem]{Proposition}
\newtheorem{corollary}[theorem]{Corollary}
\newtheorem{remark}[theorem]{Remark}
\newtheorem{definition}[theorem]{Definition}
\newtheorem{example}[theorem]{Example}
\newenvironment{dedication}
  {
   \itshape             
   \raggedleft          
  }
  {\par 
  }

\begin{document}

\title[Properties of the Riemann-Lebesgue integrability in...]{Properties of the Riemann-Lebesgue integrability in the non-additive
case}

\begin{dedication}
dedicated to Mimmo: a Master, a Colleague, a dear Friend and a Gentleman who passed away, but lives in our hearts\\
Domenico Candeloro, \dag\, May 3, 2019\\ \mbox{~}
\end{dedication}

\author{D. Candeloro}
 \address{ Department of Mathematics and Computer Sciences, University of  Perugia 1, Via Vanvitelli - 06123, Perugia, ITALY} 
               \email{domenico.candeloro@unipg.it}  
\author{A. Croitoru}
 \address{  University ''Alexandru Ioan Cuza'', Faculty of Mathematics, Bd. Carol I,
No. 11, Ia\c{s}i, 700506, ROMANIA}
 \email{croitoru@uaic.ro} 
\author{ A. Gavrilu\c{t}} 
\address{ University ''Alexandru Ioan Cuza'', Faculty of Mathematics, Bd. Carol I,
No. 11, Ia\c{s}i, 700506, ROMANIA}
 \email{gavrilut@uaic.ro} 
\author{ A.  Iosif }
\address{Petroleum-Gas University of Ploie\c{s}ti, Department of Computer Science,
          Information Technology, Mathematics and Physics, Bd. Bucure\c{s}ti, No. 39,
          Ploie\c{s}ti 100680, ROMANIA}
          \email{emilia.iosif@upg-ploiesti.ro}
\author{ A. R. Sambucini}
\address{Department of Mathematics and Computer Sciences, University of  Perugia 1, Via Vanvitelli - 06123, Perugia, ITALY}
 \email{anna.sambucini@unipg.it}

\maketitle

\begin{abstract}
We study Riemann-Lebesgue integrability of a vector function relative to an arbitrary non-negative set
function. We obtain some classical integral properties. Results regarding the continuity properties of the integral and
relationships among Riemann-Lebesgue,  Birkhoff simple and Gould integrabilities are also established.
\end{abstract}
{\noindent \small {\bf Key words} Riemann-Lebesgue integral, 
Birkhoff simple integral, Gould integral,  Non-negative
set function,  Monotone measure.\\
{\bf AMS class}: 28B20, 28C15, 49J53}
\section{Introduction}\label{intro}
Riemann and Lebesgue integrals admit different extensions, some of
them to the case of finitely additive, non-additive or set valued
measures.
Although (countable) additivity is one of the most important notion
in measure theory, it can be useless in many
problems, for example modeling different real aspects in data mining, computer
science, economy, psychology, game theory, fuzzy logic, decision
making, subjective evaluation. 
Non-additive measures have been used with different names: 
 J.von Neumann and O. Morgenstern  called them cooperative games  in Economics; 
M. Sugeno called a non additive set function a fuzzy measure and proposed an 
integral with respect to a fuzzy measure useful for systems science.
 In \cite{torra} a broad overview of non-additive measures and their applications is presented. 
So, in general,  non-additive measures are used when models based on the additive ones are not appropriate and that motivated us to study some integral properties in the non-additive case.

 In the framework of an extension of the Riemann and Lebesgue integrals, a way of defining a new integral is
to generalize the Riemann sums as in
\cite{Bir,CCG,DM,GP,Ga,Gav,GIC,G,dmm16,
K,K1,Pap,cmn17,Pap1,P,PC,PGC,PS,Si,SB,Spa,bcs2014}.
One of these extensions has been defined by Birkhoff \cite{Bir}, using
countable sums for vector functions with respect to complete
finite measures. Then this integral has been intensively studied and generalized in \cite{BM,BS1,bmis,CDPMS1,CDPMS2,CCG1,CR,CR1,CG,CS1,
FMNR,F,M,Me,Po,Po1,R1,R2}. It is
well-known that the Birkhoff integral and the (generalized) McShane
integral lie strictly between the Bochner and Pettis integrals.
The notions of Bochner and Birkhoff integrabilities are equivalent
for bounded single-valued functions with values in a separable
Banach space. Also, if the range of the Banach-valued function is separable, then
Pettis, Birkhoff and McShane integrabilities coincide (see also \cite{dp-p}). The
situation changes deeply in the non separable case, namely: every
Birkhoff integrable function is McShane integrable and every
McShane integrable function is Pettis integrable, but none of the
reverse implications holds in general (see \cite{R3}).

Another generalization belongs to Kadets and Tseytlin \cite{K}, who
introduced two integrals, called absolute Riemann-Lebesgue ($|RL|$)
and unconditional Riemann-Lebesgue ($RL$), for functions with
values in a Banach space, relative to a countably additive
measure. In \cite{K}, it is proved that if $(T,\mathcal{A},m)$ is a finite
measure space, then the following implications hold:

$f$ is Bochner integrable $\Rightarrow$ $f$ is $|RL|$-integrable
$\Rightarrow$ $f$ is $RL$-integrable $\Rightarrow$ $f$ is Pettis
integrable.

If $X$ is a separable Banach space, then $|RL|$-integrability
coincides with Bochner integrability and $RL$-integrability
coincides with Pettis integrability. Also, according to \cite{Pap1}, if
$(T,\mathcal{A},m)$ is a $\sigma$-finite measure space, then the Birkhoff
integrability \cite{Bir} is equivalent to the $RL$-integrability.

In this paper, we study the Riemann-Lebesgue integral with respect
to an arbitrary set function, not necessarily  countably additive.
The paper is organized in four sections: Section 1 is
devoted to the introduction;  after, in  Section 2, some basic concepts and results are presented,while in
Section 3  the Riemann-Lebesgue integrability relative to an
arbitrary set function is presented togheter with some classical (Theorem \ref{add-int}) and continuity properties (Theorem \ref{properties}) of the integral. Finally Section 4
contains some comparative results among this Riemann-Lebesgue integral and some other integrals known in the literature as the  
Birkhoff simple and the  Gould integrabilities. In particular we show that the Gould and the Riemann-Lebesgue integrabilities coincide when a bounded function is integrated with respect to a $\sigma$-additive (or monotone and $\sigma$-subadditive) measure of finite variation, see in this regard Proposition \ref{4.10} and Theorem \ref{part1}; lastly an example (Example \ref{esempio}) is given in order to show that the equivalence, in general, does not hold.

\section{Preliminaries}\label{prelim}
Let $T$ be a nonempty at least countable set, $\mathcal{P}(T)$ the family of all subsets of 
$T$ and $\mathcal{A}$ a $\sigma $-algebra of subsets of $T.$ Let $\mathbb{N}^{*}=\{1, 2, 3,\ldots\}.$ 

\begin{definition}\label{2.1}
\rm \quad
\begin{itemize}
\item[(\ref{2.1}.i)] A finite (countable, resp.) partition of $T$ is a finite
(countable, resp.) family of nonempty sets $P=\{A_{i}\}_{i=1, \ldots, n}$ ($\{A_{n}\}_{n\in \mathbb{N}}$, resp.) $\subset \mathcal{A}$
such that $A_{i}\cap A_{j}=\emptyset ,i\neq j$ and $\bigcup
\limits_{i=1}^{n}A_{i}=T$ ($\bigcup\limits_{n\in \mathbb{N}}A_{n}=T$,
resp.)$. $ 

\item[(\ref{2.1}.ii)] If $P$ and $P^{\prime }$ are two 
partitions of $T$, then $P^{\prime }$ is said to be \textit{finer than} $P$,
denoted\ by $P\leq P^{\prime }$ (or $P^{\prime }\geq P)$, if every set of $
P^{\prime }$ is included in some set of $P$. 

\item[(\ref{2.1}.iii)] The \textit{common refinement} of two  finite or countable
partitions $P=\{A_{i}\}$ and $P^{\prime }=\{B_{j}\}$ is the partition $
P\wedge P^{\prime }=\{A_{i}\cap B_{j}\}$.
\end{itemize}
\end{definition}
We denote by $\mathcal{P}$\textrm{ the class of all the partitions of }$T$
\textrm{and if }$\mathrm{
A\in }$ $\mathcal{A}$\textrm{\ is fixed, by }$\mathcal{P}_{\mathrm{A}}$
\textrm{\ we denote the class of all the partitions of }$\mathrm{A.}$

\begin{definition}\label{2.2}
\rm
  Let $m :\mathcal{A}\rightarrow \lbrack 0,+\infty )$
be a non-negative function, with $m (\emptyset )=0.$
$m $ is said to be: 
\begin{itemize}
\item[(\ref{2.2}.i)]
 \textit{monotone} if $m (A)\leq m (B)$, $\forall A,B\in
\mathcal{A}$, with $A\subseteq B$; 
\item[(\ref{2.2}.ii)]  \textit{\ subadditive} if $m (A\cup B)\leq m (A)+ m (B),$
for every $A,B\in \mathcal{A}$, with $A\cap B=\emptyset ;$ 
\item[(\ref{2.2}.iii)]  \textit{a submeasure } (in the sense of Drewnowski)
if $m $ is monotone and subadditive; 
\item[(\ref{2.2}.iv)] $\sigma $\textit{-subadditive} if $m (A)\leq
\sum\limits_{n=0}^{+\infty }m (A_{n}),$ for every sequence of (pairwise
disjoint) sets\textit{\ }$(A_{n})_{n\in \mathbb{N}}\subset $ \textit{$%
\mathcal{A}$, }with $A=\bigcup\limits_{n=0}^{+\infty }A_{n}$. 
\end{itemize}

\end{definition}
 Let $(X,\|\cdot\|)$ be a Banach space and $m :\mathcal{A}\rightarrow X$
be a vector set function, with $m (\emptyset )=0.$
\begin{definition}\label{2.2a}\rm 
A vector set function $m $ is said to be:
\begin{itemize} 
\item[(\ref{2.2a}.i)] \textit{finitely additive} if $m (A\cup B)=m (A)+ m (B),$
for every disjoint sets $A,B\in \mathcal{A};$ 
\item[(\ref{2.2a}.ii)]   \textit{$\sigma $-additive} if $m
(\bigcup\limits_{n=0}^{+\infty }A_{n})=\sum\limits_{n=0}^{+\infty} m (A_{n})$
, for every sequence of pairwise disjoint sets $(A_{n})_{n\in \mathbb{N}
}\subset \mathcal{A}$; 
\item[(\ref{2.2a}.iii)]   \textit{order-continuous} (shortly, \textit{o-continuous}) if $
\lim\limits_{n\rightarrow +\infty }m (A_{n})=0,$ for every decreasing
sequence of sets $(A_{n})_{n\in \mathbb{N}}\subset \mathcal{A}$, with $\bigcap\limits_{n=0}^{+\infty}A_{n}= \emptyset$ (denoted by
$A_{n}\searrow \emptyset $); 
\item[(\ref{2.2a}.iv)]   \textit{exhaustive }if $\lim\limits_{n\rightarrow +\infty }m
(A_{n})=0,$ for every sequence of pairwise disjoint sets $(A_{n})_{n\in
\mathbb{N}}\subset \mathcal{A}.$ 
\item[(\ref{2.2a}.v)]  
 {\em null-additive} if, for every $A, B \in \mathcal{A}$, $m(A \cup B) = m(A)$ when $m(B) = 0$.
\end{itemize}
Moreover $m$ 
satisfies property $(\sigma)$ if 
 the ideal of $m$-zero sets is stable under countable unions, namely
 for every sequence $(A_k)_k$ such that $m(A_k) = 0$ for each $k \in \enne$, then also
$m(\bigcup _k A_k) = 0$.

\end{definition}

\begin{definition}\label{var}
 \rm  
 The variation $\overline{m }$ of $m$ is the set
function $\overline{m }:\mathcal{P}(T)\rightarrow \lbrack 0,+\infty ]$
defined by $\overline{m }(E)=\sup \{\sum\limits_{i=1}^{n}\|m (A_{i})\|\}$,
for every $E\in \mathcal{P}(T)$, where the supremum is extended over all
finite families of pairwise disjoint sets $\{A_{i}\}_{i=1}^{n}\subset
\mathcal{A}$, with $A_{i}\subseteq E$, for every $ i=1, \ldots, n$.\\
 $m $ is said to be \textit{of finite variation} (on $\mathcal{A%
}$) if $\overline{m }(T)<+\infty$.\\ 
$\widetilde{m}$ $:\mathcal{P}(T)\rightarrow \lbrack
0,+\infty ]$ is defined for every $A\subseteq T$, by
$
\widetilde{m }(A)=\inf \{\overline{m }(B);A\subseteq B,B\in \mathcal{A}
\}$.
\end{definition}

\begin{remark}\label{rem-var}
\rm 
 $\overline{m}$ is monotone and super-additive on $\mathcal{P}%
(T)$, that is $\overline{m}(\bigcup_{i\in I}A_{i})\geq
\sum_{i\in I}\overline{m}(A_{i}),$ for every finite or countable
partition $\{A_{i}\}_{i\in I}$ of $T$.\\
 If $m:\mathcal{A}\rightarrow [0,+\infty)$ is finitely additive, then $\overline{m}(A)=m(A),$
for every $A\in \mathcal{A}.$
\\
 If $m:\mathcal{A}\rightarrow [0,+\infty) $ is subadditive ($\sigma$-subadditive,
respectively), then $\overline{m}$ is finitely additive ($\sigma $-additive, respectively). 
\end{remark}

For all unexplained definitions, see for example \cite{CCG1,CCG2}.
\section{Riemann-Lebesgue integrability with respect to an arbitrary set function}\label{RLint}

In this section, we study Riemann-Lebesgue integrability of vector
functions with respect to an arbitrary non-negative set function,
and point out some classic properties and continuity properties of
the integral.
In what follows, 
$m:\mathcal{A}\to [0,+\infty)$ is a non-negative
set function.\\

As in \cite[Definition 4.5]{K} (for scalar functions) and \cite[Definition 7]{Po} and \cite{K}, (for vector functions),  we introduce the following definition: 

\begin{definition}\label{3.1}
\rm A  function $f:T\to X$ is called
\textit{absolutely (unconditionally} respectively) \textit{Riemann-Lebesgue} ($|RL|$) ($RL$ respectively)
\textit{$m$-integrable} (on $T$) if there exists $b\in X$
such that for every $\varepsilon>0$, there exists a countable partition
$P_{\varepsilon}$ of $T$, so that for every other countable partition
$P=\{A_{n}\}_{n\in \mathbb{N}}$ of $T$ with $P\geq P_{\varepsilon}$,
$f$ is bounded on every $A_{n}$, with $m(A_{n})>0$ and for every $t_{n}\in A_{n}$, $n\in \mathbb{N}$,
the series $\sum_{n=0}^{+\infty}f(t_{n})m(A_{n})$ is absolutely (unconditionally respectively)
convergent and $\|\sum_{n=0}^{+\infty}f(t_{n})m(A_{n})-b\|<\varepsilon.$
\end{definition}
The vector $b$ (necessarily unique) is called \textit{the absolute (unconditional)
Riemann-Lebesgue} $m$\textit{-integral of} $f$ \textit{on} $T$ and it
is denoted by ${\scriptstyle (|RL|)}\displaystyle{\int_T}fdm$ (${\scriptstyle (RL)}\displaystyle{\int_T}fdm$ respectively).

The $|RL|$ ($RL$ respectively) definitions of the integrability on a subset $A\in \mathcal{A}$ are defined in the classical manner. 
Moreover, the next characterization holds:
\begin{theorem}\label{3.4}
For a vector function $f: T \to X$, the following properties hold:
\begin{itemize}
\item[\ref{3.4}.a)]
If $f$  is $|RL|$ $m $-integrable on $T$, then  $f$ is $|RL|$ $m $-integrable on  every   $A\in
\mathcal{A}$;
\item[\ref{3.4}.b)]  $f$ is $|RL|$ $m $-integrable on  every   $A\in
\mathcal{A}$ if and only if 
 $f\chi _{A}$ is $|RL|$ $m $-integrable on $T$.
 In this case, ${\scriptstyle (|RL|)}\dint_{A}fdm ={\scriptstyle (|RL|)}\dint_T f\chi _{A}dm .$
 \end{itemize}
 (The same holds for $RL$-integrability).
\end{theorem}

\begin{proof}
As soon as $f$ is $|RL|$ (RL) integrable in $T$, then it is integrable as well on every subset $A\in \mathcal{A}$. Indeed, fix any $A\in \mathcal{A}$ and
denote by $J$ the integral of $f$; then, fixed any $\varepsilon>0$,  there exists a partition $P_{\varepsilon}$ of $T$, such that, for every finer partition $P'$, $P'=\{A_{n}\}_{n\in \mathbb{N}}$ it is
$$\left\|\sum_{n=0}^{+\infty}f(t_n)m(A_n)-J \right\|\leq \varepsilon.$$
Now, 
denote by $P_0$ any partition  finer than $P$ and also finer than $\{A, T\setminus A\}$, and by $P_A$ the partition 
of $A$ consisting of all the elements of $P_0$ that are contained in $A$. Next, let $\Pi_A$ and $\Pi'_A$ denote two
 partitions of $A$ finer than $P_A$, and extend them with a common partition of $T\setminus A$ (also with the
 same {\em tags}) in such a way that the two resulting partitions, denoted by $\Pi$ and $\Pi'$, are both finer 
than $P$. So,  if we denote by
$\sigma(f,P):=\sum_{n=0}^{+\infty}f(t_{n})m(A_{n}), \, \{A_n \in \Pi\}$, 
$$\|\sigma(f,\Pi)-\sigma(f,\Pi')\|\leq
\|\sigma(f,\Pi)-J\|+\|J-\sigma(f,\Pi')\|\leq 2\varepsilon.$$
Now, setting:
$$\alpha_1:=\sum_{I\in \Pi_A}f(t_I)m(I),\hskip.6cm \alpha_2:=\sum_{I\in \Pi'_A}f(t'_I)m(I), \hskip.6cm 
\beta:=\sum_{I\in \Pi, I\subset A^c} f(\tau_I)m(I),$$ 
(with the obvious meaning of the symbols), one has
$$2\varepsilon\geq\|\alpha_1 + \beta - (\alpha_2+\beta)\|=\|\alpha_1-\alpha_2\|.$$
By the arbitrariness of $\Pi_A$ and $\Pi'_A$, this means that the sums $\sigma(f,\Pi_A)$ satisfy a Cauchy 
principle in $X$, and so the first claim follows by completeness.
\\
  Now, let us suppose that $f$ is $|RL|$ $m $-integrable on 
$A\in \mathcal{A}$. Then for every $\varepsilon >0$ there exists a partition $
P_{A}^{\varepsilon }\in \mathcal{P}_{A}$ so that for every partition $
P_{A}=\{B_{n}\}_{n\in \mathbb{N}}$ of $A$ with $P_{A}\geq P_{A}^{\varepsilon
}$ and for every $s_{n}\in B_{n},n\in \mathbb{N}$, we have
\begin{eqnarray}\label{1}
\left\|\sum_{n=0}^{+\infty}f(s_{n})m (B_{n})-{\scriptstyle (|RL|)}\int_{A}fdm \right\|<\varepsilon.
\end{eqnarray}
Let us consider 
$P_{\varepsilon }=P_{A}^{\varepsilon }\cup \{T\setminus A\}$
, which is a partition of $T$. If $P=\{A_{n}\}_{n\in \mathbb{N}}$ is a partition of $T$ with $P\geq
P_{\varepsilon }$, then without any loss of generality we may write 
$
P=\{C_{n},D_{n}\}_{n\in \mathbb{N}}$ with pairwise disjoint $C_{n},D_{n}$
such that $A=\displaystyle\cup _{n=0}^{+\infty }C_{n}$ and $\displaystyle\cup
_{n=0}^{+\infty }D_{n}=T\setminus A.$ Now, for every $u_{n}\in A_{n},n\in \mathbb{N}$ we get by (\ref{1}):
\begin{eqnarray*}
&&\left\|\sum_{n=0}^{+\infty}f\chi _{A}(u_{n})m
(A_{n})- {\scriptstyle (|RL|)}\int_{A}fdm \right\|=\\
= && \left\|\sum _{n=0}^{+\infty}f\chi _{A}(t_{n})m
(C_{n})+\sum\limits_{n=0}^{+\infty}f\chi _{A}(s_{n})m (D_{n})-{\scriptstyle (|RL|)}\int_{A}fdm
\right\|=\\
=&&\left\|\sum _{n=0}^{+\infty}f(t_{n})m (C_{n})-{\scriptstyle (|RL|)}\int_{A}fdm \right\|<\varepsilon ,
\end{eqnarray*}
where $t_{n}\in C_{n}, s_{n}\in D_{n},$ for every $n\in \mathbb{N},$
which says that $f\chi _{A}$ is $|RL|$ $m $
-integrable on $T$ and $\displaystyle{{\scriptstyle (|RL|)}\int_T}f\chi _{A}dm ={\scriptstyle (|RL|)}\int_{A}fdm .$\\

Finally, suppose that $f\chi _{A}$ is $|RL|$ $m $-integrable on $T$. Then for every $
\varepsilon >0$ there exists $P_{\varepsilon }=\{B_{n}\}_{n\in \mathbb{N}}\in
\mathcal{P}$ so that for every $
P=\{C_{n}\}_{n\in \mathbb{N}}$ partition of $T$ with $P\geq P_{\varepsilon }$
and every $t_{n}\in C_{n},n\in \mathbb{N}$, we have
\begin{eqnarray}\label{2}
\left\|\sum_{n=0}^{+\infty}f\chi _{A}(t_{n})m (C_{n})- 
{\scriptstyle (|RL|)}\int_T f\chi
_{A}dm \right\|<\varepsilon.
\end{eqnarray}
Let us consider $P_{A}^{\varepsilon }=\{B_{n}\cap A\}_{n\in \mathbb{N}}$,
which is a partition of $A$. Let $P_{A}=\{D_{n}\}_{n\in \mathbb{N}}$
be an arbitrary partition of $A$ with $P_{A}\geq P_{A}^{\varepsilon }$ and 
$P=P_{A}\cup \{T\setminus A\}.$ Then $P\in \mathcal{P}$ and $P\geq
P_{\varepsilon }.$ Let us take $t_{n}\in D_{n},n\in \mathbb{N}$ and $s\in
T\setminus A$. By (\ref{2}) we obtain
\begin{eqnarray*}
&&\left\|\sum_{n=0}^{+\infty}f(t_{n})m (D_{n})-\displaystyle{{\scriptstyle (|RL|)}\int_T}f\chi _{A}dm
\right\|=\\
= && \left\|\sum_{n=0}^{+\infty}f\chi _{A}(t_{n})m (D_{n})+f\chi _{A}(s)m
(T\setminus A)-\displaystyle{{\scriptstyle (|RL|)}\int_T}f\chi _{A}dm \right\|<\varepsilon ,
\end{eqnarray*}
which assures that $f$ is $|RL|$ $m $-integrable on $A$ and 
$${\scriptstyle (|RL|)}\int_{A}fdm ={\scriptstyle (|RL|)}\int_T f\chi _{A}dm.$$
\end{proof}

 If $X$ is finite dimensional, then $|RL|$
$m$-integrability is equivalent to $RL\, m$-integrability. In this
case, it is called $m$-integrability for short and the integral is denoted by $\displaystyle{\int_T f dm}$.\\

 We observe also that the $RL$-integral is stronger than the Birkhoff integral (in the sense of Fremlin), and it has been examined,  in the countably additive case, by  Potyrala in \cite{Po}.
 \\

 Obviously, if $f$ is $|RL|$ $m$-integrable, then it is also RL $m$-integrable.

\begin{proposition}\label{upper-bound}
 If $f:T \to X$ is bounded and 
$\overline{m}(T) < +\infty$, then $f$ is $|RL|$ $m$-integrable and
$$\left\|{\scriptstyle (|RL|)}\int_T f dm \right\|\leq \sup\limits_{t\in T}\|f(t)\|\cdot \overline{m}(T).$$
\end{proposition}
\begin{proof}
The series $\sum _{n=0}^{+\infty}f(t_{n})m(A_{n})$
is absolutely convergent.
Indeed, for every $n\in \mathbb{N}$, we have $$\sum_{k=0}^{n}\|f(t_{k})m(A_{k})\|\leq \sup\limits_{t\in T}\|f(t)\|\cdot \overline{m}(T),$$
which means that the series 
$\sum_{n=0}^{+\infty}f(t_{n})m(A_{n})$ is absolutely convergent. 
Then clearly, if $f$ is $|RL|$ $m$-integrable, we have 
$$\left\|{\scriptstyle (|RL|)}\int_T f dm \right\|\leq \sup\limits_{t\in T}\|f(t)\|\cdot \overline{m}(T).$$
\end{proof}
\begin{theorem}\label{trepuntotre}
Let $\overline{m}(T) < +\infty$  and let $f:T\rightarrow X$ be bounded. If $f$ $=0$ $m$-a.e., then $f$ is  $|RL|$ $m$-integrable and $\displaystyle{{\scriptstyle (|RL|)}\int_T}fdm = 0$.
\end{theorem}
\begin{proof}
 Since $f$ is bounded, then there exists $M\in \lbrack 0,+\infty )$
such that $\|f(t)\|\leq M$, for every $t\in T$.
If $M=0$, then the conclusion is obvious.  Suppose $M>0.$ Let us denote by $A$ the set $A:=\{t\in
T;f(t)\neq 0 \}$. Since $f=0$ $m$-a.e., we have $\widetilde{m}
(A)=0 $. Then, for every $\varepsilon >0$, there exists $B_{\varepsilon }\in
\mathcal{A}$ so that $A\subseteq B_{\varepsilon }$ and $\overline{m }
(B_{\varepsilon })<\varepsilon /M.$ Let us take the partition $
P_{\varepsilon }=\{C_{n}\}_{n\in \mathbb{N}}$ of $B_{\varepsilon}$,
and let $C_{0}=T\setminus B_{\varepsilon}$ and  add $C_0$ to $P_{\varepsilon }$. \\
Let us consider an arbitrary partition $P=\{A_{n}\}_{n\in \mathbb{N}}$ of $T$ such that $P\geq
P_{\varepsilon }$. \\ Without any loss of generality, we suppose that $P=\{D_{n},E_{n}
\}_{n\in \mathbb{N}}\subset \mathcal{A},$ with pairwise disjoint sets $D_{n},E_{n}$ such that $
\bigcup_{n\in \mathbb{N}}D_{n}=C_{0}$ and $\bigcup_{n\in
\mathbb{N}}E_{n}=B_{\varepsilon }.$
Let  $t_{n}\in D_{n}, s_{n}\in E_{n}$, for every $n\in \mathbb{N}.$
Then we can write
\begin{eqnarray*}
&& 
 \|\sum\limits_{n=0}^{+\infty}f(t_{n})m (D_{n})+\sum\limits_{n=0}^{+\infty}f(s_{n})m
(E_{n})\|=\|\sum\limits_{n=0}^{+\infty}f(s_{n})m (E_{n})\|\leq  \\
 &&\leq
 \sum\limits_{n=0}^{+\infty}\|f(s_{n})\|m (E_{n})\leq
 M\cdot \overline{m }
(B_{\varepsilon })<\varepsilon,
\end{eqnarray*}

which ensures that $f$ is $|RL|$ $m $-integrable and $\displaystyle{{\scriptstyle (|RL|)}\int_T}f dm =0.$ 
\end{proof}

The following properties hold:

\begin{theorem}\label{add-int}
Let $f,g:T\rightarrow X$ be  $|RL|$ $m $-integrable functions and $\alpha \in\mathbb{R}$. Then:
\begin{description}
\item[\ref{add-int}.a)] $\alpha f$ is $|RL|$ $m $-integrable and
$
\displaystyle{{\scriptstyle (|RL|)}\int_T}\alpha fdm =\alpha \displaystyle{{\scriptstyle (|RL|)}\int_T}fdm .
$

\item[\ref{add-int}.b)] $f$  is $|RL|$ $\alpha m $-integrable (for $\alpha \in
\lbrack 0,+\infty ))$ and
$$
\displaystyle{{\scriptstyle (|RL|)}\int_T}fd(\alpha m )=\alpha \displaystyle{{\scriptstyle (|RL|)}\int_T}fdm .
$$
\item[\ref{add-int}.c)]  $f \pm g$ is $|RL|$ $m $-integrable and 
\begin{eqnarray}\label{3}
\displaystyle{{\scriptstyle (|RL|)}\int_T}(f \pm g)dm = \displaystyle{{\scriptstyle (|RL|)}\int_T}fdm
\pm \displaystyle{{\scriptstyle (|RL|)}\int_T}gdm.
\end{eqnarray}
\end{description}
 A similar result holds also for the $RL$-integrability.
\end{theorem}

\begin{proof} 
Statements \ref{add-int}.a), \ref{add-int}.b) are straightforward. Let us prove \ref{add-int}.c) for the sum.
Let $\varepsilon >0$ be arbitrary. Since $f$ is $|RL|$ $
m $-integrable, then there exists a countable partition $P_{1}\in \mathcal{P}$ so that for every other countable partition $P\in \mathcal{P}
,P=\{A_{n}\}_{n\in \mathbb{N}},$ with $P\geq P_{1}$ and every $t_{n}\in
A_{n},n\in \mathbb{N}$, the series $\sum_{n=0}^{+\infty}f(t_{n})m (A_{n})$
and  $\sum_{n=0}^{+\infty}g(t_{n})m (A_{n})$
 are absolutely convergent
and
\begin{eqnarray*}
&&
\left\|\sum_{n=0}^{+\infty}f(t_{n})m (A_{n})-\displaystyle{{\scriptstyle (|RL|)}\int_T}fdm \right\| <\dfrac{\varepsilon }{2};\\
&&
\left\| \sum_{n=0}^{+\infty}g(t_{n})m (A_{n})-\displaystyle{{\scriptstyle (|RL|)}\int_T}gdm \right\| <\dfrac{\varepsilon }{2}.
\end{eqnarray*} 
Then, for every countable partition
$P=\{C_{n}\}_{n\in \mathbb{N}}\in \mathcal{P},$ with $P\geq P_{1}$ and every
$t_{n}\in C_{n},n\in \mathbb{N}$, the series $\sum_{n=0}^{+\infty}(f+g)(t_{n})m (C_{n})$
is absolutely convergent and 
 we get
\begin{eqnarray*}
&&
\left\|\sum_{n=0}^{+\infty}(f+g)(t_{n})m (C_{n})-\left(\displaystyle{{\scriptstyle (|RL|)}\int_T}fdm
+\displaystyle{{\scriptstyle (|RL|)}\int_T}gdm \right)  \right\|\leq\\ 
\leq &&  \left\|\sum_{n=0}^{+\infty}f(t_{n})m (C_{n})-\displaystyle{{\scriptstyle (|RL|)}\int_T}fdm \right\|
+ \left\|\sum_{n=0}^{+\infty}g(t_{n})m (C_{n})-\displaystyle{{\scriptstyle (|RL|)}\int_T}gdm  \right\| \leq \\
\leq && \dfrac{\varepsilon }{2}+\dfrac{\varepsilon }{2}=\varepsilon .
\end{eqnarray*}
Hence $f+g$ is $|RL|$ $m $-integrable and (\ref{3}) is satisfied.
\end{proof}
We point out that the previous result holds without any additivity condition on $m$. In particular,
for every $f:T\rightarrow X$ that is $|RL|$  ($RL$)  $m$-integrable on every set $A\in \mathcal{A}$, let $I_f :\mathcal{A}\rightarrow X$ defined
by 
\begin{eqnarray}\label{If}
I_f  (A)={\scriptstyle (|RL|)}\int_{A}f \, dm \quad \left( I_f  (A)={\scriptstyle (RL)}\int_{A}f \, dm \right).
\end{eqnarray}
We have:
\begin{corollary}\label{cortreotto}
Let $f:T\rightarrow X$ be any $|RL|$  (resp. $RL$) $m$-integrable function.
Then $I_f$ is finitely additive.
\end{corollary}
\begin{proof}
It follows from Theorem \ref{3.4} 
and Theorem \ref{add-int}.
\end{proof}

\begin{corollary}\label{pigreco}
Suppose $\overline{m}(T) < +\infty$.
Let $f,g:T\rightarrow X$ be vector functions with $\sup\limits_{t\in
T} \|f(t)-g(t)\|<+\infty $, $f$ is $|RL|$ $m $-integrable and $
f=g $ $m $-a.e.. Then $g$ is $|RL|$ $m $-integrable and 
$$\displaystyle{{\scriptstyle (|RL|)}\int_T}fdm
=\displaystyle{{\scriptstyle (|RL|)}\int_T}gdm .$$
\end{corollary}

\begin{proof} 
It is enough to observe that $g-f$ satisfies the hypotheses of Proposition \ref{trepuntotre}, and so it is $|RL|$-integrable, with null integral. Then $g=(g-f)+f$ is integrable as well, and its integral coincides with ${\scriptstyle (|RL|)}\displaystyle{\int}_T f dm$ by Theorem \ref{add-int}.
\end{proof}

By a similar reasoning, as in Theorem \ref{add-int}.c), we get the following result:
\begin{theorem}\label{3.9}
Let $m _{i} :\mathcal{A}\rightarrow [0,+\infty )$
be a non-negative set function,
$i=1,2$. 
Suppose
$f:T\rightarrow X$ is both $|RL|$ (resp. RL)  $m _{1}$-integrable and $|RL|$ (resp. RL)  $m
_{2}$-integrable. Then $f$ is $|RL|$ (resp. RL)  $m_1+m_2 $-integrable and 
$$
{\scriptstyle (|RL|)}\int_T f d(m_{1}+m_{2})=
{\scriptstyle (|RL|)}\int_T f dm_{1}+{\scriptstyle (|RL|)}\int_T f dm_{2}.$$
\end{theorem}

\begin{theorem}\label{3.10}
Let $\overline{m}(T) <+\infty$. If $f,g:T\rightarrow X$ are $|RL|$ $m $%
-integrable functions, then 
\begin{eqnarray*}
\left\| \displaystyle{{\scriptstyle (|RL|)}\int_T}fdm - {\scriptstyle (|RL|)}\int_{T}gdm \right\| \leq \sup_
{t\in T}\|f(t)-g(t)\|\cdot \overline{m }(T).
\end{eqnarray*}
\end{theorem}

\begin{proof}
 If $\sup_{t\in T}\|f(t)-g(t)\|=+\infty$,
then the conclusion is obvious. 
If $\sup_{t\in T}\|f(t)-g(t)\|<+\infty$
 the conclusion follows directly from Theorem \ref{add-int} and Proposition \ref{upper-bound},  applied to $f-g$.
\end{proof}

\begin{theorem}\label{monotonicity}
If $f,g:T\rightarrow \mathbb{R}$ are $m $-integrable functions
such that $f(t)\leq g(t),$ for every $\mathit{t\in T,}$ then $\displaystyle{\int_T}fdm \leq \displaystyle{\int_T}gdm $. 
\end{theorem}
\begin{proof}
 Let $\varepsilon >0$ be arbitrary. Since $f,g$ are $
m $-integrable, there exists $P_{0}\in \mathcal{P}$ such that for every
$P=\{C_{n}\}_{n\in \mathbb{N}}\in \mathcal{P},P\geq P_{0}$ and
every $t_{n}\in C_{n},n\in \mathbb{N}$, the series $\sum_{n=0}^{+\infty}f(t_{n})m (C_{n})$,   $\sum_{n=0}^{+\infty}g(t_{n})m (C_{n})$
are 
absolutely convergent and
\begin{eqnarray*}
\left|\displaystyle{\int_T} fdm -\sum\limits_{n=0}^{+\infty}f(t_{n})m (C_{n}) \right|<\frac{\varepsilon}{3};
\\
\left|\displaystyle{\int_T}gdm -\sum_{n=0}^{+\infty}g(t_{n})m (C_{n}) \right|<\frac{\varepsilon}{3}.
\end{eqnarray*}
Therefore
\begin{eqnarray*}
&&  \int_Tfdm -\displaystyle{\int_T}gdm =
 \int_T fdm
-\sum_{n=0}^{+\infty}f(t_{n})m (C_{n}) +\\+  &&
\sum_{n=0}^{+\infty}f(t_{n})m
(C_{n})-\sum_{n=0}^{+\infty}g(t_{n})m
(C_{n}) +\sum_{n=0}^{+\infty}g(t_{n})m (C_{n})-\displaystyle{\int_T}gdm 
<  \\< &&
\dfrac{2\varepsilon }{3}+[\sum_{n=0}^{+\infty}f(t_{n})
m (C_{n})-\sum_{n=0}^{+\infty}g(t_{n})m (C_{n})] \leq \varepsilon
\end{eqnarray*}
since, by the hypothesis,  $\sum_{n=0}^{+\infty}f(t_{n})m
(C_{n})\leq \sum_{n=0}^{+\infty}g(t_{n})m (C_{n}).$\\
 Consequently, $
\displaystyle{\int_T}fdm -\displaystyle{\int_T}gdm \leq 0.$ 
\end{proof}

\medskip
We analogously obtain the following theorem.

\begin{theorem}\label{monot2}
\textit{Let $m _{1}$, $m _{2}:\mathcal{A}\rightarrow
[0, +\infty )$ be  set functions such that $m _{1}(A)\leq m
_{2}(A)$, for
every $A\in \mathcal{A}$ and $f:T\rightarrow [0, \infty)$
a simultaneously $m _{1}$-integrable and $m _{2}$-integrable
function. Then $\displaystyle{\int_T}fdm _{1}\leq
\displaystyle{\int_T}fdm _{2}$. }
\end{theorem}

In the next theorem we point out some results concerning the properties of the set
function $I_f$.

\begin{theorem}\label{properties}
Let $f:T\rightarrow X$ be a vector function such that $f$
is $|RL|$ $m $-integrable on $T$.
Then the following properties hold:
\begin{itemize}
\item[\rm \ref{properties}.i)] If  $f$ is bounded, and $\overline{m}(T) < +\infty$, then 
\begin{itemize}
\item[\rm \ref{properties}.i.a)]  $I_f \ll \overline{m} $ (in the $\varepsilon$-$\delta$ sense)
and $I_f$ is of finite variation.
\item[\rm \ref{properties}.i.b)]  If moreover 
 $\overline{m}$ is o-continuous (exhaustive respectively), then $I_f$ is also o-continuous (exhaustive respectively).
\end{itemize}
\item[\rm \ref{properties}.ii)]  If $f:T \rightarrow [0,\infty)$  and $m $ is  scalar-valued and monotone, then the same holds for $I_f$.

\end{itemize}
\end{theorem}
\begin{proof} 
\ref{properties}.i.a)  The absolute continuity of $I_f$ in the $\varepsilon$-$\delta$ sense follows from 
Proposition \ref{upper-bound}.
In order to prove that $I_f$ is of finite variation let $\{A_{i}\}_{i=1, \ldots, n}\subset \mathcal{P}(T)$ be pairwise
disjoint sets and $M=$ $\underset{t\in T}{\sup }\|f(t)\|.$ By  Proposition \ref{upper-bound}
 and Remark \ref{rem-var}, 
it follows 
$$\sum_{i=1}^{n}\|I_f(A_{i})\|\leq M\cdot\sum_{i=1}^{n}
\overline{m }(A_{i})\leq M\cdot\overline{m }(T).$$
 This implies $\overline{
I_f}(T)\leq M$ $\overline{m }(T)$,
which ensures that $I_f$ is of finite variation.\newline
\ref{properties}.i.b) Let $M=\sup_{t\in T}\|f(t)\|.$ If $M=0$, then $f=0$, hence $I_f=0.$\\
Let us suppose that $M>0$. According to 
Proposition \ref{upper-bound} 
we have $\|I_f(A)\|\leq M\cdot \overline{m}(A),$
for every $A\in \mathcal{A}.$ If $\overline{m}$ is o-continuous, it follows that $I_f$ is also o-continuous.
The proof is similar when $\overline{m}$ is exhaustive. \\
\ref{properties}.ii)
 Let be $A, B\in \mathcal{A}$ with $A\subseteq B$ and $\varepsilon >0$. Since $f$ is $%
m $-integrable on $A$, there exists a countable partition $P_{1}=\{C_{n}\}_{n\in \mathbb{N}}\in \mathcal{P}_{A}$ so that for every
 other countable partition $P=\{A_{n}\}_{n\in \mathbb{N}}\in \mathcal{P}_{A},P\geq P_{1}$ and
every $t_{n}\in A_{n},n\in \mathbb{N}$, the series $\sum\limits_{n=0}^{+\infty}f(t_{n})m (A_{n})$ is
absolutely convergent and
\begin{eqnarray}\label{11}
\left|I_f(A) -\sum_{n=0}^{+\infty}f(t_{n})m (A_{n})\right|<\frac{\varepsilon}{2}.
\end{eqnarray}

Analogously, since $f$ is $m $-integrable on $B$, there exists a countable partition
 $P_{2}=\{D_{n}\}_{n\in \mathbb{N}}\in \mathcal{P}_{B}$
 such that for every other countable partition
$P=\{B_{n}\}_{n\in \mathbb{N}}\in \mathcal{P}_{B},P\geq P_{2}$ 
and every $
t_{n}\in B_{n},n\in \mathbb{N}$, the series $\sum\limits_{n=0}^{+\infty}f(t_{n})m (B_{n})$
is absolutely convergent and
\begin{eqnarray}\label{12}
\left|I_f(B) -\sum_{n=0}^{+\infty}f(t_{n})m (B_{n})\right|<\frac{\varepsilon}{2}.
\end{eqnarray}
Consider $\widetilde{P}_{1}= \{C_{n}, B\setminus A\}_{n\in \mathbb{N}}$.
Then $\widetilde{P}_{1}\in \mathcal{P}_{B}$ and $\widetilde{P}_{1}\wedge P_{2}\in \mathcal{P}_{B}.$  Let $P=\{E_{n}\}_{n\in\mathbb{N}}\in \mathcal{P}_{B}$ be
an arbitrary countable partition, with $P\geq \widetilde{P}_{1}\wedge P_{2}$.
We observe that $P^{''}=\{E_{n}\cap A\}_{n\in \mathbb{N}}$ is also a partition of $A$ and $P^{''}\geq P_{1}.$
If $t_{n}\in E_{n}\cap A, n\in \mathbb{N}$, by (\ref{11}) and (\ref{12}), we have $|I_f(B) 
- \sum\limits_{n=0}^{+\infty}f(t_{n})m (E_{n})|<\dfrac{\varepsilon }{2}$ and $
|I_f(A) -\sum\limits_{n=0}^{+\infty}f(t_{n})m (E_{n}\cap A)|<\dfrac{\varepsilon}{2}$. Therefore
\begin{eqnarray*}
&&
 I_f(A) - I_f(B) \leq \left| I_f(A)
-\sum\limits_{n=0}^{+\infty}f(t_{n})m (E_{n}\cap A)\right|+
\\ &+& 
\left[\sum\limits_{n=0}^{+\infty}f(t_{n})m
(E_{n}\cap A)-\sum\limits_{n=0}^{+\infty}f(t_{n})m
(E_{n}) \right]+ \left|\sum\limits_{n=0}^{+\infty}f(t_{n})m (E_{n})-
I_f(B) \right|< \\
&<&
\varepsilon+ \left[\sum\limits_{n=0}^{+\infty}f(t_{n})%
m (E_{n}\cap A)-\sum\limits_{n=0}^{+\infty}f(t_{n})m (E_{n}) \right].
\end{eqnarray*}
From the hypothesis, it results 
$\sum\limits_{n=0}^{+\infty}f(t_{n})m
(E_{n}\cap A)\leq \sum\limits_{n=0}^{+\infty}f(t_{n})m (E_{n}).$
Consequently, 
$I_f(A) - I_f(B) \leq\varepsilon $, 
for every $\varepsilon >0$, which implies
 $I_f(A) \leq I_f(B) .$
  \end{proof}

\section{Comparative results}\label{comp}
In this section, comparison results among Riemann-Lebesgue,
 and Birkhoff simple \cite{CCG1}
integrabilities are presented.
\begin{definition}\label{4.3}
\rm (\cite[Definition 3.2]{CCG1})
A vector function $f:T\rightarrow X$ is called \textit{Birkhoff
simple $m$-integrable (on $T$)} if there exists $a\in X$ such that
for every $\varepsilon>0$, there exists a countable partition $P_{\varepsilon}$ of
$T$ so that for every other countable partition $P=\{A_{n}\}_{n\in
\mathbb{N}}$ of $T$, with $P\geq P_{\varepsilon}$ and every $t_{n}\in
A_{n}, n\in \mathbb{N},$ it holds
$\limsup_{n\rightarrow +\infty}\|\sum_{k=0}^{n}f(t_{k})m (A_{k})- a\|<\varepsilon.$\\
The vector $a$ is denoted by $(Bs)\displaystyle{\int_T}fdm$ and it is called \textit{the Birkhoff simple integral}
of $f$ (on $T$) with respect to $m$.
\end{definition}

\begin{theorem}\label{4.5}
If $f:T\rightarrow X$ is RL $m$-integrable, then
$f$ is also Birkhoff simple $m$-integrable and the two integrals coincide.
\end{theorem}

\begin{proof} Let $\varepsilon>0$ be arbitrary and let  $P_{\varepsilon}$ the countable
partition of $T$ that one given by Definition \ref{3.1}. Considering
any countable partition $P=\{A_{n}\}_{n\in \mathbb{N}}$ of $T$,
with $P\geq P_{\varepsilon}$ and $t_{n}\in A_{n}, n\in \mathbb{N}$, by
Definition \ref{3.1}, it results that the series
$\sum\limits_{n=0}^{+\infty}f(t_{n})m (A_{n})$ is unconditionally
convergent and
\begin{eqnarray*}
 \limsup_{n\rightarrow +\infty}
 && 
 \left\|\sum\limits_{k=0}^{n}f(t_{k})m (A_{k})-  {\scriptstyle (|RL|)}\int_T fdm \right\|=
 \\ =&&
  \left\|\sum\limits_{n=0}^{+\infty}f(t_{n})m (A_{n})-
{\scriptstyle (|RL|)}\int_T fdm \right\|<\varepsilon,
\end{eqnarray*}
which implies the Birkhoff simple $m$-integrability of $f$ and
$(Bs)\displaystyle{\int_T}fdm= {\scriptstyle (|RL|)}\displaystyle{\int_T}fdm.$
 \end{proof}

Next, we highlight some comparison results between Riemann-Lebesgue and Gould integrabilities. 
We first recall the Gould integrability for vector functions (called RLF-integrability in \cite{Po}).\\

Let $(X, \|\cdot\|)$ be a Banach space. If $f:T\rightarrow X$ is a vector function, we denote by $\sigma (P)$ the sum $\sigma (P):=
 \sum\limits_{i=1}^{n}f(t_{i})m (A_{i}),$ for every finite partition of $T$, $P=\{A_{i}\}_{i=1, \ldots, n}\in \mathcal{P}$
 and every $t_{i}\in A_{i}, i=1,\ldots, n.$
\begin{definition}\label{4.8}
\rm (\cite{G})
A vector function $f:T\rightarrow X$ is called \textit{Gould $m$-integrable} (on $T$) if the
net $(\sigma(P))_{P\in (\mathcal{P},\leq)}$ is convergent in $X$, 
 where the family 
$\mathcal{P}$ of all {\em finite} partitions of $T$ is ordered by the relation "$\leq$" given in Definition \ref{2.1}.ii). The limit of
$(\sigma(P))_{P\in (\mathcal{P},\leq)}$ is called \textit{the Gould integral
of $f$ with respect to $m$}, denoted by $(G)\displaystyle{\int_T}fdm.$
\end{definition}
\begin{remark}\label{4.9}
\rm
$f$ is \textit{Gould $m$-integrable} if and only if there exists
$a\in X$ such that for every $\varepsilon>0$, there exists a finite partition
$P_{\varepsilon}$ of $T$, so that for every other finite partition $P=\{A_{i}\}_{i=1, \ldots, n}$
of $T$, with $P\geq P_{\varepsilon}$ and every $t_{i}\in A_{i}, i=1, \ldots, n,$ we have
$\|\sigma(P)-a\|<\varepsilon.$
\end{remark}
\begin{proposition}\label{4.10}
  Let $m:\mathcal{A}\rightarrow [0, +\infty)$ be a complete $\sigma$-additive measure of
finite variation, and $f:T\rightarrow \mathbb{R}$ a bounded
function. Then $f$ is Gould $m$-integrable if and only if $f$ is
Riemann-Lebesgue $m$-integrable and the two integrals coincide.
\end{proposition}
\begin{proof}
It is an easy consequence of \cite[Theorem 5.1]{CG} and of \cite[Theorem 8]{Po}.
\end{proof}
\begin{example}\label{esempio}\rm 
We show now that, without assuming additivity of $m$, the previous equivalence does not hold. Suppose that $T=\enne$, with $\mathcal{A}=\mathcal{P}(\enne)$,
and define the scalar measure $m$ as follows:
$$m(A)=\left\{\begin{array}{ll}
0, \ & card(A)<+\infty\\
1, &  card(A)=+\infty.\end{array}\right.$$
Then, the constant function $f(n)\equiv 1$ is clearly simple-Birkhoff integrable, with null integral, and also RL integrable with the same integral: it suffices to take as partition $P_{\varepsilon}$ the finest possible one, consisting of all singletons.
\\
However, this mapping is not Gould-integrable. In fact, if $P_{\varepsilon}$ is any finite partition of $\enne$, some of its sets are infinite, so the quantity $\sigma(P_{\varepsilon})$ coincides with the number of the infinite sets of $P_{\varepsilon}$. And the  quantity $\sigma(P)$ is unbounded when  $P$ runs over the family of all finer partitions of  $P_{\varepsilon}$.
\end{example}

\begin{theorem}\label{part1}
 Suppose $m:\mathcal{A}\rightarrow [0, +\infty)$ is monotone, $\sigma$-subadditive and of finite variation and let $f:T\rightarrow \mathbb{R}$ be
a bounded function. Then $f$ is Gould $m$-integrable if and only
if $f$ is Riemann-Lebesgue $m$-integrable.\\ In this case,
$(G)\displaystyle{\int_T}fdm= {\scriptstyle (|RL|)}\displaystyle{\int_T}fdm.$
\end{theorem}
\begin{proof} Since $m$ is a submeasure of finite
variation, according to  \cite[Theorem 2.15]{GP}, $f$ is Gould
$m$-integrable if and only if
$f$ is Gould $\overline{m}$-integrable and $(G)\displaystyle{\int_T}fdm=
(G)\displaystyle{\int_T}fd\overline{m}.$ On the other hand, from Remark \ref{rem-var}   
and Proposition \ref{4.10},
Gould $\overline{m}$-integrability  is
equivalent to Riemann-Lebesgue and the two integrals coincide.
 Now the conclusion follows by \cite[Theorem 5.5]{CG}. 
\end{proof}

In the general case only partial results can be obtained, for example:
\begin{theorem}\label{part2}
Suppose that $m:\mathcal{A}\rightarrow [0, +\infty)$ satisfies property ($\sigma$),  is
monotone, null-additive set function and $A\in \mathcal{A}$ is an atom of $m$. If a bounded function
$f:T\rightarrow \mathbb{R}$ is Riemann-Lebesgue $m$-integrable on $A$, then
$f$ is Gould $m$-integrable on $A$ and $(G)\displaystyle{\int_A}fdm={\scriptstyle (|RL|)}\displaystyle{\int_A}fdm.$
\end{theorem}
\begin{proof} It follows in an analogous way as the previous theorem, using in this case
\cite[Theorem 5.2]{CG}.
\end{proof}
\small
\section*{Acknowledgements}
The first and the last  authors have been partially supported by University of Perugia -- Department of Mathematics and Computer Sciences-- Ricerca di Base 2018 "Metodi di Teoria dell'Approssimazione, Analisi Reale, Analisi Nonlineare e loro applicazioni"  and the last author  by  GNAMPA - INDAM (Italy) "Metodi di Analisi Reale per l'Approssimazione attraverso operatori discreti e applicazioni" (2019).

\end{document}